%% file: Stokes_pressure.tex
\documentclass[a4paper,11pt, reqno]{amsart} 
\usepackage[margin=2cm]{geometry} 
\pdfoutput=1

\usepackage[utf8]{inputenc}
\usepackage[T1]{fontenc}
\usepackage{lmodern}
\usepackage[french,english]{babel}
\usepackage{amsthm, amssymb, amsmath, amsfonts, mathrsfs, dsfont, esint, textcomp, bm}
\usepackage[colorlinks=true, pdfstartview=FitV, linkcolor=blue, citecolor=blue, urlcolor=blue,pagebackref=false]{hyperref}
\usepackage{mathtools, enumitem}
\usepackage{tikz}

\usepackage{microtype}

\usepackage[notcite,notref,color,final]{showkeys}
\definecolor{labelkey}{gray}{.8}
\definecolor{refkey}{gray}{.8}

\definecolor{darkblue}{rgb}{0,0,0.7} 
\definecolor{darkred}{rgb}{0.9,0.1,0.1}
\definecolor{darkgreen}{rgb}{0,0.5,0}

\newtheorem{thm}{Theorem}[section]

\newtheorem{lem}[thm]{Lemma}

\theoremstyle{remark}
\newtheorem{rem}[thm]{Remark}
\theoremstyle{definition}

\renewcommand{\leq}{\leqslant}
\renewcommand{\geq}{\geqslant}

\renewcommand{\subset}{\subseteq}

\newcommand{\I}{\mathcal{I}}

\newcommand{\uh}{u_{\mathsf{h}}}
\newcommand{\ph}{p_{\mathsf{h}}}
\newcommand{\F}{\mathcal{F}}

\newcommand{\N}{\mathbb{N}}

\newcommand{\R}{\mathbb{R}}

\renewcommand{\P}{\mathbb{P}}

\newcommand{\eps}{\varepsilon}

\renewcommand{\div}{\mathop{\rm div}}

\newcommand{\aeps}{\eps^\frac{d}{d-2}}

\newcommand{\wto}{\rightharpoonup}
\newcommand{\Rd}{\R^d}

\newcommand{\rr}{\mathcal{R}}
\newcommand{\km}{k_{max}}
\newcommand{\supp}{\text{supp}}

\DeclareMathOperator{\dist}{dist}
\DeclareMathOperator{\dv}{div}

\title[]{Convergence of the pressure in the homogenization of the Stokes equations in randomly perforated domains}
\author{Arianna Giunti, Richard M. H\"ofer}

\begin{document}
%

%


%
%
%
%
%
%
%
%
%

\input{Introduction.tex}

\input{MainResult.tex}


\input{pressure.tex}



\section{Appendix}\label{appendix}

\input{appendix.tex}

\end{document}

%% file: Introduction.tex
\begin{abstract}
We consider the homogenization to the Brinkman equations for the incompressible Stokes equations in a bounded domain which is perforated by a random
collection of small spherical holes. This problem has been studied by the same authors in [A.~{Giunti} and R.M.~{H\"ofer}, \textit{Homogenization for
  the Stokes equations in randomly perforated domains under almost minimal assumptions
  on the size of the holes}] where convergence of the fluid velocity field towards the solution of the Brinkman equations has been established. 
  In the present we consider the pressure associated to the solution of the Stokes equations in the perforated domain.
  We prove that it is possible to extend this pressure inside the holes and slightly modify it in a region of asymptotically negligible harmonic
   capacity such that it weakly converges to the pressure associated with the solution of the Brinkman equations.
\end{abstract}
\maketitle


\section{Introduction}
\label{sec:intro}
In this paper we consider the steady incompressible Stokes equations
\begin{align}
	\label{Stokes}
\begin{cases}
-\Delta u_\eps + \nabla p_\eps = f \ \ \ &\text{in $D^\eps$}\\
\nabla \cdot u_\eps = 0 \ \ \ &\text{in $D^\eps$}\\
u_\eps = 0 \ \ \ &\text{on $\partial D^\eps$}
\end{cases}
\end{align}
in a domain $D^\eps$, that is obtained by removing from a bounded set $D \subset \Rd$, $d > 2$, a random number of small balls having random centres and 
radii. More precisely, for $\eps > 0$, we define
\begin{align}
	\label{perforation}
D^\eps= D \backslash H^\eps, \ \ \ \ \ H^\eps:= \bigcup_{z_i \in \Phi \cap \frac{1}{\eps}D} B_{\aeps \rho_i} (\eps z_i),
\end{align}
where $\Phi$ is a Poisson point process on $\Rd$ with homogeneous intensity rate $\lambda > 0$, and the radii $\{\rho_i \}_{z_i \in \Phi} \subset \R_+$ are identically and
independently distributed unbounded random variables which satisfy
\begin{align}
	\label{power.law}
 \langle \rho^{(d-2)+ \beta} \rangle < +\infty, \ \ \  \text{for some $\beta > 0$.}
\end{align}  
In \cite{GH19withoutpressure}, we show that for almost every realization of $H^\eps$ in \eqref{perforation}, the solution  $u_\eps \in H^1_0(D^\eps; \Rd)$ to \eqref{Stokes} weakly converges in $H^1_0(D; \Rd)$ to the solution $\uh$ of the Brinkman equations
\begin{align}
	\label{Brinkman}
\begin{cases}
-\Delta \uh + \mu \uh + \nabla \ph  = f \ \ \ &\text{in $D$}\\
\nabla \cdot \uh = 0 \ \ \ &\text{in $D$}\\
\uh= 0 \ \ \ &\text{on $\partial D$}.
\end{cases}
\end{align}
Here,
\begin{align}\label{strange.term.Stokes}
\mu= C_d \lambda \langle \rho^{d-2} \rangle \mathrm I,
\end{align}
where $\langle \cdot \rangle$ denotes the expectation under the probability measure on the radii $\rho_i$, and the constant $C_d > 0$ depends only on the dimension $d$.  In this paper, we  give a convergence result for the families of pressures $\{ p_\eps\}_{\eps > 0}$ in \eqref{Stokes}. This result has been announced in our previous paper, 
\cite[Remark 2.3]{GH19withoutpressure}.

\medskip

There are many results in the literature dealing with the previous homogenization problem when the domain $D^\eps$ is perforated by periodic or random holes \cite{AllaireARMA1990a, Brillard1986-1987, CarrapatosoHillairet2020, Desvillettes2008, Hillairet2018, HillairetMecherbet2020,  Hillairet2017, MarKhr64, Rubinstein1986, SanchezP82}. However, the setting of \cite{GH19withoutpressure} and of the present paper is the only one allowing for the presence of ``many'' holes overlapping with overwhelming probability. By means of strong law of large numbers (see also \cite[Appendix C]{GH19withoutpressure}) it is easy to see that condition \eqref{power.law} only implies that, with overwhelming probability, the number of overlapping balls in $D^\eps$ is less than $\eps^{-(d-2)}$ (over a total number of $\eps^{-d}$). For a detailed discussion on the literature studying the homogenization of \eqref{Stokes} to \eqref{Brinkman} we refer to \cite{GH19withoutpressure}. We also mention that, in a similar probabilistic setting for the distribution of the holes $H^\eps$, an homogenization result for the Poisson problem is shown in \cite{GHV_Poisson}.

\medskip

With the exception of the case of the periodically perforated domains studied in \cite{AllaireARMA1990a}, the convergence of the pressure for \eqref{Stokes} is not considered. 
Both the Stokes and the Brinkman equations may be reformulated so that the pressure only plays the role of a Lagrange multiplier for the incompressibility of the fluid.
In particular, by testing the equations only with divergence-free functions, it is possible to obtain convergence results for the velocity field $u_\eps$ without any bound on the pressures $p_\eps$. This approach has also been used in our previous work \cite{GH19withoutpressure}. As a physical quantity, though, the pressure is important in itself. From an application oriented point of view, it is therefore desirable to obtain convergence results not only for the fluid velocity field but also for the pressure.

\medskip

Obtaining bounds for some extension of $p_\eps$ in $L^2(D)$ that are uniform in $\eps$ is usually a challenging problem. For solutions $(u,p)$ of a Stokes system in a general domain $D$ with no-slip boundary conditions, the standard energy estimate only provides an estimate for the velocity field $u$. Classically, estimates for the pressure $p$ are then obtained with the help of a Bogovsky operator that maps functions $g \in L^q(D)$ into vector fields $v \in H^q (D; \Rd)$ satisfying  $\dv v = g$ and $v = 0$ in $\partial D$. The norm of the previous operator, though, might strongly depend on the geometry of the set $K$. Hence, an immediate application of this method for each $p_\eps$ in $D^\eps$ does not yield a priori a uniform bound in $\eps$. We recall, indeed, that by our assumptions \eqref{power.law}, many balls in $H^\eps$ may overlap and give rise to many clusters having very different geometrical properties. 

\medskip

In \cite{AllaireARMA1990a}, for periodically distributed holes, a  suitable extension $P_\eps (p_\eps)$ for $p_\eps$ is obtained such that  $P_\eps p_\eps \rightharpoonup \ph$ in $L^2_0(D)$. We remark that, as $u_\eps \rightharpoonup \uh$  in $H^1_0(D ; \Rd)$, this is the optimal result that one could expect. In this paper, the problem of finding a good extension for the pressures, which is uniformly bounded in $L^2$ is dealt with by a duality argument. This, in particular, allows to define such extensions provided the construction of a good ``reduction'' operator, this time mapping vector fields  $v\in H^1(D; \Rd)$ into vector fields $v \in H^1_0(D^\eps; \Rd)$. In our setting, namely in the case of overlapping holes, the same challenges mentioned above for the construction of a Bogovski operator do arise also with this method. 
%
%

\medskip

In the current paper, we show that, at the expense of removing from the domain $D$ a set $E^\eps \supset D^\eps$ which is only slightly bigger than $H^\eps$, we may construct a Bogovski operator $D \setminus E^\eps$ that is uniformly bounded in $\eps$ (cf. Lemma \ref{lem:Bogovski}). The set $E^\eps$ is only slightly bigger than $H^\eps$ in the sense that the harmonic capacity of the difference $E^\eps \setminus D^\eps$ vanishes at infinity. From this result, we may construct a function $\tilde p_\eps \in L^2(D)$, which satisfies $\tilde p_\eps = p_\eps$ outside of $E^\eps$, and such that  $\tilde p_\eps \wto \ph$ in $L^q$, for $q< \frac{d}{d-1}$. We remark that $\tilde p_\eps$ is not a proper extension for $p_\eps$, as it might differ from $p_\eps$ on the (small) set $E^\eps \setminus H^\eps$. 

\medskip


The covering $E^\eps$ of $H^\eps$ is constructed in such a way that the Bogovski operator for $D \setminus E_\eps$ may be obtained by an iterative application of Bogovski operators on annuli.  We stress that \eqref{power.law} rules out the occurrence of clusters made of too many balls of similar size. We emphasize, however, that it neither prevents the balls generating $H^\eps$ from overlapping, nor it implies a uniform upper bound on the number of balls of very different size which combine into a cluster (see \cite[Section 5]{GH19withoutpressure}). The covering $E^\eps \subset H^\eps$ is therefore constructed with the purpose of providing a more regular set, where clusters of balls are combined together.   We stress that we only obtain a convergence of the pressure in the sub-optimal spaces $L^q$, $1 \leq q < \frac{d}{d-1}$. However, this is enough to give an alternative proof of the convergence of the fluid velocity field $u_\eps$ to $\uh$ weakly in $H^1(D; \Rd)$ by means of oscillating test functions as done in \cite{CioranescuMurat} and \cite{AllaireARMA1990a} (See Remark \ref{rem:alternative}). 


%% file: MainResult.tex
\section{Setting and main result}\label{setting}

Let $D \subset \R^d$, $d >2$, be an open and bounded set that is star-shaped with respect to the origin. For $\varepsilon > 0$, we denote by $D^\varepsilon \subset D$ the 
domain obtained as in \eqref{perforation}, namely by setting $D^\eps = D \backslash H^\eps$ with 
\begin{equation}\label{holes}
H^\varepsilon := \bigcup_{z_j \in \Phi \cap \frac 1 \eps D} B_{\varepsilon^{\frac{d}{d-2}} \rho_j} (\varepsilon z_j).
\end{equation}
Here, $\Phi \subset \Rd$ is a homogeneous Poisson point process having intensity $\lambda > 0$ and the radii
$\rr :=\{ \rho_i \}_{z_i \in \Phi}$ are i.i.d. random variables which satisfy condition \eqref{power.law} for a fixed $\beta >0$. 

\smallskip

Throughout the paper we denote by $(\Omega, \F , \P)$ the probability space associated to the marked point process $(\Phi, \rr)$, i.e. the joint process of the centres and radii
distributed as above. We refer to our previous paper  \cite{GH19withoutpressure} for a more detailed introduction of marked point processes.

\medskip 

\subsection{Notation} \label{sec:notation}
 For a point process $\Phi$ on $\R^d$ and any bounded set $E \subset \R^d$, we define the random variables
\begin{equation} 	\label{Number.function}
\begin{aligned}
	\Phi(E)&:= \Phi \cap E, && \Phi^\eps(E):= \Phi \cap \left(\frac 1 \eps E \right), \\
  N(E) &:= \# (\Phi (E)), && N^{\eps}(E) := \# (\Phi^\eps(E)).
\end{aligned}
\end{equation}
For $\eta > 0$, we denote by  $\Phi_\eta$ a thinning for the process $\Phi$ obtained as 
\begin{align}\label{thinned.process}
\Phi_\eta(\omega):= \{ x \in \Phi(\omega) \,  \colon \, \min_{y \in \Phi(\omega), \atop y \neq x} | x- y| \geq \eta \},
\end{align}
i.e. the points of $\Phi(\omega)$ whose minimal distance from the other points is at least $\eta$. Given the process $\Phi_\eta$, we set $\Phi_\eta(E)$, 
$ \Phi_\eta^\eps(E)$, $N_\eta(E)$ and $N_\eta^\eps(E)$ for the analogues for  $\Phi_\eta$ of the random variables defined in \eqref{Number.function}.

\smallskip

 For a bounded and measurable set $E \subset \Rd$ and any $1\leq p < + \infty$, we denote 
 \begin{align}
 L^p_0(E):= \{ f \in L^p(E)  \ \colon \ \int_E f = 0 \}.
 \end{align}
As in \cite{GH19withoutpressure}, we identify any $v \in H^1_0(D^\eps)$ with the function $\bar v \in H^1_0(D)$ obtained by trivially extending $v$ in $H^\eps$.

\smallskip

Throughout the proofs in this paper, we write $a \lesssim b$ whenever $a \leq C b$ for a constant $C=C(d,\beta)$ depending only on the dimension $d$ and $\beta$ from assumption
\eqref{power.law}. Moreover, when no ambiguity occurs, we use a scalar notation also for vector fields and 
vector-valued function spaces, i.e. we write for instance $C^\infty_0(D), H^1(\Rd), L^p(\Rd)$ instead of $C^\infty_0(D; \Rd), H^1(\Rd ;\Rd), L^p(\Rd;\Rd)$. Finally, given a domain $D \subset \Rd$ and a parameter $ r> 0$, we define
\begin{align}\label{r.neighbourhood}
D_r:= \{ x \in D \, \colon \, \mathop{dist}(x, \partial D) > r \}.
\end{align}

\subsection{Main result} Let $(\Phi, \rr)$ be a marked point process as above, and let $H^\varepsilon$ be defined as in \eqref{holes}. 
%
%

The main result for the pressure, which we obtain in this paper is the following.

\begin{thm} \label{t.pressure}
Let $(u_\eps, p_\eps) \in H^1_0(D^\eps; \Rd) \times L^2_0(D^\eps; \R)$ solve \eqref{Stokes} and let $(\uh, \ph) \in H^1_0(D; \Rd) \times L^2_0(D; \R)$ be the solution of the homogenized
 problem \eqref{Brinkman}. Then, for $\P$-almost every $\omega \in \Omega$ there exists a family of sets $E^\eps \subset \Rd$ and a sequence $r_\eps \to 0$ with the following properties:
\begin{enumerate}[label=(\roman*)]
 \item \label{it:E_eps} It holds $H^\eps \subset E^\eps$ and for $\eps \downarrow 0^+$
	\begin{align}
	 \operatorname{Cap}(E^\eps \backslash H^\eps) \to 0,
	\end{align}
	where $\operatorname{Cap}$ denotes the harmonic capacity in $\Rd$.
 \item Let  $D_{r_\eps}$ be as defined in \eqref{r.neighbourhood}. Then, the modification of the pressure 
	\begin{align}\label{p.tilde}
	 \tilde p_\eps = \begin{cases}
	  p_\eps - \fint_{D_{r_\eps} \backslash E^\eps} p_\eps \ \ \ \ &\text{ in $D_{r_\eps} \backslash E^\eps$}\\
	  0 \ \ \ \ &\text{ in $(D\backslash D_{r_\eps}) \cup E^\eps$}
	  \end{cases}
	\end{align}
	 satisfies for all $q < \frac{d}{d-1}$
\begin{align}
 \tilde p_\eps \rightharpoonup \ph \ \ \ \ \text{in $L^{q}_0(D ; \R)$.}
\end{align}
\end{enumerate}
\end{thm}

\subsection{Proof of the main result}\label{sec:proof.main}

The proof of Theorem \ref{t.pressure} relies on the following lemma, which is a variant of the standard Bogovski lemma to the set $D \backslash E^\eps$, with $E_\eps$ as 
in Theorem \ref{t.pressure}. This result allows to obtain estimates for the pressure in the Stokes equations \eqref{Stokes}. A priori, any such estimate highly depends on the exact geometry of the set considered.  In the following result, the specific construction of the sets $E^\eps$ allows to obtain estimates that are almost surely uniform in $\eps$.  

\medskip

\begin{lem}[\bf Estimate on the Bogovski operator in $D\setminus E^\eps$]
	\label{lem:Bogovski}
	Let $r \in (0,1)$ and $D_r$ as in \eqref{r.neighbourhood}. For $\P$-almost every $\omega \in \Omega$ there exists a family of sets $E^\eps \subset \Rd$ satisfying \ref{it:E_eps} of Theorem \ref{t.pressure}  
	and  $\eps_0= \eps_0(\omega, r) > 0$ such that for all $\eps < \eps_0$ the following holds:
	 For any $g \in L^{q}_0(D_r \setminus E^\eps)$, $q > d$, let us consider its trivial extension to  $D \setminus E^\eps$. Then, there exists $v \in H^1_0(D\setminus E^\eps)$ such that 
	\begin{equation}
	\label{Bogovski}
	\begin{aligned}
		\dv v &= g \ \ \ \text{ in $D \setminus E^\eps$} \\
		\|v\|_{H^1( D \setminus E^\eps)} & \leq C(d,\beta,q) \|g\|_{L^{q}(D_r \setminus E^\eps)}.
	\end{aligned}
	\end{equation}
\end{lem}

The previous result allows to construct suitable oscillating test functions $\{w^\eps_i\}_{i=1}^d \subset H^1(D; \Rd)$ in the same spirit of 
\cite{AllaireARMA1990a}:

\begin{lem}\label{l.oscillating.test.functions}
Let $k=1, \dots, d$ be fixed. Then, for almost every $\omega \in \Omega$ and any $\eps \leq \eps_0(\omega)$ there exists a 
set $E^\eps \subset \Rd$ such that $E^\eps \supset H^\eps$ and for $\eps \downarrow 0^+$
	\begin{align}
		\label{Cap.E^eps}
	 \operatorname{Cap}(E^\eps \backslash H^\eps) \to 0,
	\end{align}
Moreover, for all $k=1, \cdots, d$, there exist $w^{\eps}_k \in H^1(D; \Rd) \cap L^\infty(D; \Rd)$, $k = 1, \cdots d$,  such that
\begin{enumerate}[label=(H\arabic*)]
\item \label{H1} $w^\eps_k = 0$ on $E^\eps$ and $\nabla \cdot w^\eps_k =0$ in $D$;
\item \label{H2} $w^\eps_k \rightharpoonup e_k$ in $H^1(D)$ and $w^\eps_k \rightarrow e_k $ in $L^p(D)$ for any $1 \leq p< +\infty$;
\item \label{H3} For any $\phi \in C^\infty_0(D)$ and sequence $v_\eps \rightharpoonup v$ in $H^1_0(D; \Rd)$ with $\nabla \cdot v_\eps =0$ on $D$ we have
\begin{align}
\lim_{\eps \downarrow 0^+} \int \phi \nabla w^\eps_k : \nabla v_\eps = \int \phi e_k \cdot \mu v,
\end{align}
with $\mu$ defined in Theorem \ref{t.pressure}.
\end{enumerate}
\end{lem}

\begin{proof}[Proof of Theorem \ref{t.pressure}]
	Testing \eqref{Stokes} with the solution $u_\eps$ itself yields the energy estimate
	\begin{align}
		\|u_\eps\|_{H_0^1(D)} \leq \|f\|_{H^{-1}(D)}.
	\end{align}
	In particular, $u_\eps \wto u^\ast$ in $H^1_0(D)$ for a subsequences and some $u^\ast \in H^1_0(D)$ with $\dv u^\ast = 0$.

	Let $\{ r_n\}_{n \in \N} \subset \R_+$ be such that $r_n \to 1$. Let us denote by $\eps_{0,n}> 0$ the minimum between the (random values) $\eps_0$ in the statements of Lemma \ref{lem:Bogovski} with $r=r_n$ and Lemma \ref{l.oscillating.test.functions}.
	By construction, we may assume that $\eps_{n + 1} \leq \eps_n$ for all $n \in \N$; we thus define
	$r_\eps :=r_n$, for $\eps_{n+1} \leq \eps < \eps_n$.
	Let $q > d$ and $g \in L^{q}_0(D_{r_\eps} \setminus E^\eps)$ and let $v \in H^1_0(D\setminus E^\eps)$ satisfy \eqref{Bogovski}.
	Then, testing \eqref{Stokes} with $v$ yields
	\begin{align} \label{eq:test.w_k}
		\int_{D_{r_\eps} \setminus E^\eps} p_\eps g = \int_{D \setminus E^\eps} p_\eps \dv v &= (\nabla u_\eps, \nabla v)_{L^2(D^\eps)} + \langle f, v \rangle_{H^{-1}, H^1} \\
		&\leq 2 \|v\|_{H^1} \|f\|_{H^{-1}} \leq  C(d,\beta,q) \|g\|_{L^{q}} \|f\|_{H^{-1}}.
	\end{align}
	Since $g \in L^{q}(D_{r_\eps}\setminus E^\eps)$ was arbitrary, this implies that, up to a subsequence, $\tilde p_\eps$ defined in \eqref{p.tilde} converges to $p^\ast$ weakly
	in $L^{q'}(D)$, where $q'$ is the H\"older conjugate of $q$. It remains to show that $(p^\ast,u^\ast)= (\ph,\uh)$ as the unique weak solutions to
	\eqref{Brinkman}. By uniqueness of the limit, the convergence then holds for the whole family $\eps \downarrow 0^+$.

	To do so, we fix any smooth vector field $\phi \in C^\infty_0(D)$. Then $\supp \phi \subset D_{r_\eps}$ for $\eps$ sufficiently small (depending on $\omega$ and $\phi$).  We test the equation \eqref{Stokes} for $u_\eps$ with the admissible test function $\sum_{k=1}^d w_k^\eps \phi_k$. This yields
	\begin{align}
		\sum_{k=1}^d \int \nabla u_\eps : \nabla (w_k \phi_k) - \sum_{k=1}^d \int \nabla \cdot (w^\eps_k \phi_k) p_\eps 
		=  \sum_{k=1}^d \langle  w^\eps_k \phi_k,f \rangle_{H^1,H^{-1}}.
	\end{align}		
	Convergence of the right-hand side follows immediately from \ref{H2} of Lemma \ref{l.oscillating.test.functions}.
	For the first term on the left-hand side, we observe that \ref{H2} implies $w_k \to w$ strongly in $L^2(D)$. Thus, by \ref{H2} and \ref{H3}
	\begin{align}
		\lim_{\eps \downarrow 0^+} \int \nabla u_\eps : \nabla (w_k \phi_k) &=  
		\lim_{\eps \downarrow 0^+} \int \nabla u_\eps : w_k \otimes \nabla \phi_k + \int \nabla u_\eps : \nabla w_k  \phi_k \\
		&=\int \nabla u^\ast : e_k \otimes \nabla \phi_k + \int \phi_k e_k \cdot\mu u^\ast.
	\end{align}
 
We turn to the second term on the left-hand side of \eqref{eq:test.w_k}. By \ref{H1} of Lemma \ref{l.oscillating.test.functions} each product $w^\eps_k \phi_k$ is supported in 
	$D_{r_\eps} \backslash E_\eps$ and therefore
	\begin{align}
	\int \nabla \cdot (w^\eps_k \phi_k) p_\eps =  \int \nabla \cdot (w^\eps_k \phi_k) \tilde p_\eps = \int w^\eps_k \cdot \nabla \phi_k \tilde p_\eps,
	\end{align}
where in the last identity we used Leibniz rule and the divergence-free condition for $w_k^\eps$ in \ref{H1} of Lemma \ref{l.oscillating.test.functions}.
It now remains to combine the convergence of $\tilde p_\eps$ with \ref{H2} of Lemma \ref{l.oscillating.test.functions} and send $\eps \downarrow 0^+$ in the
right-hand side above. This yields
\begin{align}
	\lim_{\eps \downarrow 0^+}  \int \nabla \cdot (w^\eps_k \phi_k) p_\eps = \int e_k \cdot \nabla \phi_k p^\ast.
\end{align}
Combining the above identities yields
\begin{align}
	\int \nabla u^\ast : \nabla \phi + \int \phi \cdot\mu u^\ast - \int  \nabla \cdot \phi p^\ast
	= \langle   \phi,f \rangle_{H^1,H^{-1}},
\end{align}
which is the weak formulation of \eqref{Brinkman}.
\end{proof}

\begin{rem}\label{rem:alternative}
We point out the the argument used in Theorem \ref{t.pressure} allows to deduce both the convergence of the velocities $u_\eps$ and the pressures $p_\eps$, i.e. both \cite[Theorem 2.1]{GH19withoutpressure}  and Theorem \ref{t.pressure} of the current paper. We also remark that Lemma \ref{l.oscillating.test.functions} is an adaptation of  \cite[Lemma 2.5]{GH19withoutpressure} to the case when the reduction operator $R_\eps$ defined in that lemma (see also Subsection \ref{sub:ideas}) is applied to the functions $\varphi = e_i$, $i=1, \cdots d$. Here, $e_k$, $k=1, \cdots, d$ are the canonical basis vectors of $\Rd$. We stress that we may not immediately use \cite[Lemma 2.5]{GH19withoutpressure} as the vectors $e_i$ are not in $C^\infty_0(D; \Rd)$. 
\end{rem}

%% file: pressure.tex
\section{Proof of Lemma \ref{lem:Bogovski}} \label{s.pressure}

\subsection{Strategy for the proof of Lemma \ref{lem:Bogovski}}\label{sub:ideas}
As described in \cite{GH19withoutpressure}, the main challenge in our problem is related to the geometry of the holes $H^\eps$. The homogenization result for the velocities
$u_\eps \in H^1_0(D^\eps)$ of \cite[Theorem 2.1]{GH19withoutpressure}, relies indeed on the construction of a ``reduction'' operator
$R_\eps : \{ v \in C_0^{\infty}(D) \, \colon \, \dv v = 0 \} \to H^1_0(D^\eps ; \R^d)$ that transforms smooth vector fields on $D$ into admissible test functions for
\eqref{Stokes} and preserves the divergence-free condition (see \cite[Lemma 2.5]{GH19withoutpressure}). In order to have good bounds for $R_\eps$ that are deterministic and independent from $\eps$, we need to construct $R_\eps$
in such a way that it depends on the geometry of the set $H^\eps$ in a uniform way (in $\eps$ and $\omega \in \Omega$). It is easy to imagine that the main challenge is
given by the subset of $H^\eps$ made of holes that overlap giving rise to clusters of holes with various possible geometries.  

\medskip

We tackle this issue by constructing a covering $\bar H^\eps$ of $H^\eps$ that allows us to construct $R^\eps$ by solving a finite number of iterated boundary value problems.

\medskip

Roughly speaking, the covering $\bar H^\eps$ is obtained by selecting only certain balls of $H^\eps$ and dilating them by a factor $\lambda^\eps_j \leq \Lambda$, with $\Lambda$ 
finite and deterministic. In other words, by considering the set 
$$
\bar H^\eps = \bigcup_{z_j \in J^\eps} B_{\eps^{\frac{d}{d-2}} \lambda^\eps_j \rho_j}(\eps z_j), \ \ \ \ J^\eps \subset \Phi(\frac{1}{\eps}D).
$$

The main property of this covering is the following: we may find a suitable finite partition  $J_1, \cdots, J_{\km}$ of $J$ such that 
$\bar H^\eps:= \bigcup_{k \in \N} \bigcup_{j \in J_k} B_{\eps^{\frac{d}{d-2}}\lambda^\eps_j \rho_j}(\eps z_j)$ and, given a vector field $\varphi \in C^\infty_0(D; \R^d)$ having
zero divergence, define $R_\eps \varphi$ by iteratively correcting $\varphi$ in such a way that it vanishes on the sets
$\bigcup_{z_j \in J_k} B_{\eps^{\frac{d}{d-2}} \lambda^\eps_j \rho_j}(\eps z_j)$, $k=1, \cdots, \km$ and it preserves the divergence-free condition. In other words, we define 
$\varphi^{(1)}, \cdots \varphi^{(\km)}$ in such a way that each $\varphi^{(k)}$ is obtained from $\varphi^{(k-1)}$ by adding suitable corrections in order for $\varphi^{(k)}$ to vanish on 
$\bigcup_{z_j \in J_k} B_{\eps^{\frac{d}{d-2}} \lambda^\eps_j \rho_j}(\eps z_j)$.

\medskip

On the one hand, the balls $B_{\eps^{\frac{d}{d-2}} \lambda^\eps_j \rho_j}(\eps z_j)$ belonging to the same collection $J_k$ are mutually disjoint, even if dilated by a certain
factor $\theta$; this allows us to define the corrections at each step by solving independent Stokes problems in annuli. On the other hand, balls belonging to different
subcollections $J_k \neq J_i$ may overlap. However, their intersection is disjoint from the set of ``real'' holes $H^\eps$. These conditions imply that if $E_k$ and
$E_{k-1}$ are the sets where $\varphi^{(k)}$ and $\varphi^{(k-1)}$ vanish, respectively, then $E_k$ and $E_{k-1}$ are \textit{not nested}. Nonetheless, the set $E_{\km}$ obtained in 
the final iteration does contain the full set of holes $H^\eps$, i.e. the final function $\varphi^{(\km)}$ \textit{does} vanish on $H^\eps$. For a more precise discussion, we refer to \cite[Subsection 2.3]{GH19withoutpressure}.

\medskip

The proof of Lemma \ref{lem:Bogovski} relies on an idea similar to the one for $R_\eps$ and the sets $E_{k}$, $k=1 , \cdots, \km$ defined above play a curcial role also
in this proof. In particular, the set $E^\eps$ corresponds to $E^{\km}$, i.e. the set where the operator $R_\eps$ above  satisfies
$R_\eps \varphi = 0$ for all divergence free $\varphi \in C_0^\infty(D)$.

\medskip

In fact, given the function 
$g \in L^{q}_0(D_r \backslash E^\eps)$ we need to define a suitable extension $\bar g$ to the whole set $D$ in such a way that:
\begin{itemize}
\item  There exists a solution $v_0 \in H^1_0(D)$ such that
		\begin{equation}
	\label{v_0.0.stra}
	\begin{aligned}
		\dv v_0 &= \bar g, \\
		\|v_0\|_{H^{1}} &\lesssim \|g\|_{L^{q}(D_r \backslash E^\eps)}.
	\end{aligned}
	\end{equation}
\item We may modify $v_0$ by adding corrections in such a way that we achieve the boundary condition $v=0$ in $E^\eps$. Similarly to what is done in \cite[Lemma 2.5]{GH19withoutpressure}
for $R_\eps$, we add correctors in a recursive way by constructiong functions $v^{(1)}, \cdots, v^{(\km)}$ such that each $v^k$ satisfies $v^k=0$ in $E_k$ and 
$\div v = \bar g$ in $D \setminus E_k$. We remark that the fact that the sets $E_k$ are not nested implies that at each step we have to restore the condition
$\div v =\bar g$ in $E_{k+1}\setminus E_k$. This yields a compatibility condition which determines the extension of $\bar g$. By exploiting the properties of the sets $E_k$ and, in particular of $E^\eps$ (see Lemma \ref{l:E_-3}), we show that such an extension $\bar g$ exists and that the algorithm
for correcting the function $v_0$ is well-defined and satisfies \eqref{Bogovski}. 
\end{itemize}
\subsection{Summary of the geometric properties of $H^\eps$ proved in \cite{GH19withoutpressure}}\label{sub:geometry}
This subsection is devoted to recalling the main geometric results for the set of holes $H^\eps$ obtained \cite{GH19withoutpressure}. 
This allows us to rigorously introduce the set $E^\eps$, together with other auxiliary sets $E_{k}$ (cf. the previous subsection) that play a curcial role in the 
proof of Lemma \ref{lem:Bogovski}. Finally, Lemma \ref{l:E_-3} contains some further properties of the previous sets which were not proven in \cite{GH19withoutpressure}.

\begin{lem}[{\cite[Lemmas 3.1 and 3.2]{GH19withoutpressure}}]\label{l.geometry}
There exists $\delta = \delta(d,\beta) > 0$ such that for almost every $\omega \in \Omega$ and all $\eps \leq \eps_0=\eps_0(\omega)$, 
there exists a partition $H^\eps= H^\eps_g \cup H^\eps_b$ and a set $D^\eps_b \subset \Rd$ such that
$H^\eps_b \subset D^\eps_b$ and
\begin{align}
	\label{safety.layer}
\dist( H^\eps_g ; D^\eps_b) > \eps^{1 + \delta}.			
\end{align}
Furthermore, $H_g^\eps$ is a union of disjoint balls centred in $n^\eps \subset \Phi^\eps(D)$, namely
\begin{equation}\label{good.set}
\begin{aligned}
 H^\eps_g = \bigcup_{z_i \in n^\eps} B_{\aeps \rho_i}(\eps z_i), \ \ \ \ \eps^d \#n^\eps \to \lambda \, |D|, \\
\min_{z_i \neq z_j \in n^\eps} \eps |z_i - z_j| \geq 2 \eps^{1+\frac \delta 2}, \quad \aeps \rho_i \leq  \eps^{1+2\delta}.
\end{aligned}
\end{equation}

Moreover, let $\mathcal I^\eps := \Phi^\eps(D) \backslash n^\eps$, so 
that
\begin{align}
	\label{H^b}
H^\eps_b:= \bigcup_{z_i \in \mathcal \I^\eps} B_{\aeps \rho_i}(\eps z_i).
\end{align}
Then, there exist $ \Lambda(d, \beta)> 0$, $k_{max}= k_{max}(d, \beta)>0$, and sub-collections $J^\eps_k  \subset \mathcal I^\eps$, $-3 \leq k \leq k_{max}$ and constants $\{ \lambda_l^\eps \}_{z_l\in J^\eps} \subset [1, \Lambda]$ such that
 \begin{itemize}
 \item For all  $k=-3, \cdots, k_{max}$ and every $z_j \in J_k^\eps$ 
 	\begin{align}
 		\label{eq:order.J_k}
		\eps^{1 - \delta k} \leq \aeps \rho_i &< \eps^{1 - \delta (k+1)}  \quad \text{for } k \geq -2, \\
		\aeps \rho_i &< \eps^{1 +2\delta} \quad \text{for }k= -3.
 	\end{align}
 \item	Let $J^\eps = \cup_{k=-3}^{k_{max}} J_k^\eps$. Then,
 \begin{align}
		\label{bar.H^b}
	 H_b^\eps \subset \bar H^\eps_b := \bigcup_{z_j \in J^\eps} B_{\lambda_j^\eps \aeps \rho_j}( \eps z_j), \ \ \ \lambda_j^\eps \aeps \rho_j \leq \Lambda \eps^{2d\delta}.
	\end{align}
 \item  For all $k=-2, \cdots, k_{max}$ and every $z_i, z_j \in J_k^\eps \cup J_{k-1}^\eps$, $z_i \neq z_j$
\begin{align}\label{similar.size.apart}
B_{\theta^2 \lambda_i^\eps \aeps \rho_i}(\eps z_i) \cap B_{\theta^2 \lambda_j^\eps \aeps \rho_j}(\eps z_j) = \emptyset.
\end{align}
\item The set $D^\eps_b$ might be chosen as
\begin{align}
\label{D_b}
& D^\eps_b = \bigcup_{z_i \in J^\eps} B_{\theta \aeps \lambda_i^\eps\rho_i}(\eps z_i).
\end{align}
 \end{itemize}

%
\end{lem}
\begin{rem}
	Properties \eqref{eq:order.J_k} and \eqref{similar.size.apart} differ from the statement in \cite{GH19withoutpressure}.
	However, they follow directly from the construction in the proof, see 
	\cite[equations (3.11), (3.27) and (3.32)]{GH19withoutpressure} for \eqref{eq:order.J_k} and
	\cite[equation (3.14)]{GH19withoutpressure} for \eqref{similar.size.apart}.
\end{rem}

In order to define the set $E^\eps$, we need to recall the basic construction leading to the results above. For the sake of a leaner notation, when no ambiguity occurs we
will omit the index $\eps$ in the definitions below. We begin by introducing the notation 
\begin{equation}\label{definitions}
\begin{aligned}
	B_j &= B_{\aeps \rho_j}(\eps z_j),  \qquad B_{j,\theta} = B_{\theta \aeps \rho_j}(\eps z_j), 
	\qquad A_j = B_{j,\theta} \setminus B_j \quad \text{for } j \in n^\eps \\
	B_j & = B_{\lambda_j \aeps \rho_j}(\eps z_j)	, \qquad B_{j,\theta} = B_{\theta \lambda \aeps \rho_j}(\eps z_j), 
	\qquad A_j = B_{j,\theta}  \setminus B_j \quad \text{for } j \in J.
\end{aligned}
\end{equation}

\medskip

We recall the definition of the set $E_{k}$ for $-3 \leq k \leq k_{\max}+1$ from the proof of \cite[Lemma 3.2]{GH19withoutpressure}:
\begin{align}\label{remove.annuli}
	E_{k_{max} +1} := \emptyset, \ \ \ E_{k-1} := \biggl( E_l \setminus \bigcup_{z_j \in J_{k-1}}  A_j \biggr) \cup \bigcup_{z_j \in J_{k-1}}   B_j.
\end{align}
We now define
\begin{align}
	\label{E^eps}
	E^\eps := E_{-3} \cup H_g^\eps.
\end{align}
 
\medskip

As shown in \cite[equation (3.52)]{GH19withoutpressure}, for all $-3 \leq k \leq k_{\max}$ we may write
	\begin{align}
		\label{E_-3}
		E_{k} = \dot{\bigcup_{l \geq k}} \dot{{\bigcup_{j \in J_l}}} E^{z_i}_l,
	\end{align}	\\
with
\begin{equation}\label{E_k^z_j}
E^{z_j}_{k} := B_{j}\setminus \bigcup_{m = k}^{l-1} \bigcup_{z_i \in J_{m}}  B_{i,\theta}.
\end{equation}
Finally, we denote for any $z_j \in  \bigcup_{k = -3}^{\km} \bigcup_{z_i \in J_{k}}$
\begin{align}\label{def.E.z}
	E^{z_j} := E^{z_j}_{-3}.
\end{align}

\bigskip

The next technical lemma contains some further properties satisfied by the sets $E^{z_j}_k$ and $E^\eps_k$ defined above.  As it will become apparent in the next subsection,
the results stated below allow to tackle the main challenge in the construction of the Bogovski operator in Lemma \ref{lem:Bogovski}. We postpone the proof of this result
to the appendix.

\begin{lem} 
	\label{l:E_-3}
	For all almost every $\omega \in \Omega$ and all $\eps \leq \eps_0(\omega)$, we have
	\begin{align}
		\label{Cap.E^eps.2}
	 \operatorname{Cap}(E^\eps \backslash H^\eps) \to 0.
	\end{align}

	For every $-3 \leq k \leq \km$ and all  $z_j \in J_k$, let $B_j$ and $A_j$ be as in \eqref{definitions}. Then, there exists a constant $c=c(d) > 0$ such that
	\begin{align}
	\label{hole.comparable}
		|E^{z_j} \setminus E_{k+1}| \geq c |B_j|.
	\end{align}
	
	Furthermore,
	\begin{align}
		\label{char.E^z_j}
		E^{z_j} = E \cap B_j \setminus \biggl(\bigcup_{l=-3}^{k-2}\bigcup_{i \in J_l} E^{z_i}\biggr),
	\end{align}
	and
	\begin{align}
		\label{balls.in.annulus}
		A_j \cap E_{k+1} \cap E =  A_j \cap E_{k+1} \cap \biggl(\bigcup_{l=-3}^{k-2}\bigcup_{i \in J_l} E^{z_i}\biggr).
	\end{align}
\end{lem}

\subsection{Proof of Lemma 2.2}

For the sake of a leaner notation, we drop the index $\eps$ in all the quantities defined in Lemma \ref{lem:Bogovski} and Subsection \eqref{sub:geometry}. Furthermore, we employ the notation $\lesssim$ and $\gtrsim$ for $\leq C$ and $\geq C$ with a constant only depending on $d$, the exponent $q > d$ and the domain $D$. We define the set $E$ as in \eqref{E^eps}; by Lemma \ref{l:E_-3}, this set satisfies  \ref{it:E_eps} of Theorem \ref{t.pressure}.
	
\bigskip	
	
\noindent {\bf Step 1: Setup of the construction:}	

	Let $q >d$ and $r \in (0,1)$ be fixed. Let $g \in L^{q}_0( D_r \setminus E)$ and let us trivially extend it to $D \setminus E$. Our goal is to construct  $v \in H^1_0(D\setminus E)$
	solving  $\dv v = g$ in $D_r \setminus E$. Let $\bar g$ be any extension of $g$ to the whole set $D$. 
	

	For $\bar r := (1+r)/2$, let us assume that we may find $v_0 \in H^1_0(D_{\bar r})$ such that 
	\begin{align}\label{step.zero}
	\dv v_0 = \bar g \ \ \ \text{ on $D_{\bar r}$,} \ \ \ \| v_0 \|_{H^1(D_{\bar r})} \lesssim \| \bar g \|_{L^q(D_{\bar r})}.
	\end{align}
	Let $n^\eps$ be as in Lemma \ref{l.geometry} and, for any $z_j \in n^\eps$, $A_j$, $B_j$ and $B_{j,\theta}$ as in \eqref{definitions}. For all $z_j \in n^\eps$, let us assume that there exist $(v_j, p_j)$ solving
	\begin{align}
		\label{Stokes.annulus.divergence}
		\begin{cases}
		- \Delta v_j + \nabla p_j =0 \quad &\text{in } A_j \\
		\dv v_j =  \bar g \quad &\text{in } A_j \\
		v_j = 0  \quad &\text{on } \partial B_{j,\theta} \\
		v_j = -v_0 \quad &\text{in }B_j.
		\end{cases}
	\end{align}
	Recall that, by Lemma \ref{l.geometry} and \eqref{definitions}, $H^\eps_g := \bigcup_{z_j \in n^\eps} B_j $. Then, the function 	
	\begin{align}
		\label{eq:v_km+1}
		v^{(k_{max}+1)} := v_0 + \sum_{z_j \in n^\eps} v_j
	\end{align}
 solves 
\begin{align}
\begin{cases}
\dv v^{(k_{max}+1)} = \bar g \ \ \ &\text{in $D_{\bar r} \backslash H_g$}\\
v^{(k_{max}+1)}=0 \ \ \ &\text{in $H_g$.}
\end{cases}
\end{align}

We now want to iterate the previous precedure so that we gradually solve \eqref{Bogovski} in $D_{\bar r} \backslash E$, by gradually passing through the sets $D_{\bar r} \backslash (E_k \cup H_g)$, with $E_k$ as in  \eqref{remove.annuli}. Therefore, we need to iteratively solve, for $k= \km, \km -1, \cdots -3$, the boundary value problems
		\begin{align}
		\label{Annulus.divergence}
		\begin{cases}
		\dv v_j = g \quad &\text{in } A_j \cap E_{k+1} \\
		\dv v_j = 0 \quad &\text{in } A_j \setminus E_{k+1} \\
		v_j = 0  \quad &\text{on } \partial B_{j,\theta} \\
		v_j = -v^{(k+1)} \quad &\text{in }B_j,
		\end{cases}
	\end{align}	
	 for every $z_j \in J_k$ and define the functions
	\begin{align} \label{eq:def.v_k}
		v^{(k)} := v^{(k+1)} + \sum_{z_j \in J_k} v_j.
	\end{align}
We remark that the different equation for $\dv v_j$ in $A_j \cap E_{k+1}$ and $A_j \cap E_{k+1}$ is due to the fact that the sets $E_k$, $k= \km , \cdots, -3$ are not nested as, passing from $E_{k+1}$ to $E_{k}$, at each step we remove the intersections  $A_j \cap E_{k+1}$, $j \in J_k$. In these sets we thus need to restore the divergence condition $\dv v = \bar g$.
	
\medskip

We remark that in order to solve \eqref{step.zero}, \eqref{Stokes.annulus.divergence} and \eqref{Annulus.divergence} for every $k= \km, \cdots , -3$ and implement the previous construction, we need that the compatibility conditions 
\begin{align}\label{comp.cond.average}
\int_{D_{\bar r}} \bar g = 0,
\end{align}
\begin{align}
		\label{comp.cond.good}
		 \int_{\partial B_{j}} v_0 \cdot \nu = \int_{A_j} \bar g , \ \ \ \text{for all $z_j \in n^\eps$,}
	\end{align}
	\begin{align}
		\label{comp.cond.0}
		 \int_{A_j \cap E_{k+1}} g - \int_{\partial B_{j}} v^{(k+1)} \cdot \nu = 0, \ \ \ \text{ for all $z_j \in J_k$, for every $k = \km, \cdots, -3$}
	\end{align}
	are satisfied. Therefore, we need to find a suitable extension $\bar g$ of $g$ in $E$ such that previous identities hold. This is the goal of the next step.

\bigskip

\noindent {\bf Step 2: Extension of the function $g$:}	
	For  $z_j \in n^\eps$, we may simply choose $\bar g = 0$ in $B_j$. Indeed, this yields
		\begin{align}
		 \int_{\partial B_{j}} v_0 \cdot \nu = \int_{B_j} \dv v_0 = 0,
	\end{align}
	so \eqref{comp.cond.good} is satisfied.
	
	\medskip
	
	We now turn to \eqref{comp.cond.0}. For any $k= \km, \cdots, -3$ and $z_j \in J_k$, let $E^{z_j}$ be as in \eqref{E_k^z_j}. We define 
	\begin{align}\label{def.E.tilde}
	\tilde E^{z_j}_k := E^{z_j} \backslash E_k.
	\end{align}
	 We claim that we may choose $\bar g = g_j = \mathrm{const}$ in $\tilde E^{z_j}_k$   and $\bar g = 0$ in $E^{z_j} \cap E_k$, and that the constants $g_j$ are uniquely determined by \eqref{comp.cond.0}. Indeed, using \eqref{char.E^z_j} and \eqref{balls.in.annulus} of Lemma \ref{l:E_-3} combined with \eqref{comp.cond.0},  the constants $g_j$ are determined by the following formula:
\begin{equation} 		\label{comp.cond}
	\begin{aligned}
		0 &= \int_{A_j \cap E_{k+1}} \bar g - \int_{B_{j} \setminus E_{k+1}} \bar g \\
		&= \int_{A_j \cap E_{k+1} \backslash E} g + \sum_{l=-3}^{k-2} \sum_{z_i \in J_l} |\tilde E^{z_i} \cap A_j \cap E_{k+1}| g_{i} \\
		& - \int_{B_j \setminus (E_{k+1} \cup E)} g - |\tilde E^{z_j}| g_j - 
		\sum_{l=-3}^{k-2}\sum_{z_i \in J_l}   |\tilde E^{z_i} \cap B_j \setminus E_{k+1}| g_{i}.
	\end{aligned}
\end{equation}
	This formula indeed yields $g_j$ for all $z_j \in J_k$, provided we already know $g_i$ for $z_i \in \cup_{l=-3}^{k-1} J_l$. Therefore, all $z_j$, $j \in J$ are inductively defined by \eqref{comp.cond}. Note that by \eqref{safety.layer} and \eqref{D_b}, the balls $B_{j,\theta}$, $z_j \in n^\eps$ do not intersect with any of the balls $B_{j,\theta}$, $z_j \in J$. Therefore, the set $H_g^\eps$ does not play any role in the previous conditions and in the constructions of the solutions to \eqref{Annulus.divergence}.

	\medskip
	
	We observe that by this procedure we might extend the function $g$ non-trivially also in holes that are not contained in $D_r$,
	namely if they are within a cluster that intersects with $D_r$. This motivates the introduction of the auxiliary set $D_{\bar r}$, $r < \bar r < 1$. We remark that, for $\eps$ sufficiently small, $\bar g= 0$ in $D \setminus  D_{\bar r}$. This follows by an induction argument similar to the one at the end of Step 2 in the proof of \cite[Lemma 2.5]{GH19withoutpressure}. 
	 Indeed, $g_j = 0$ for all $j \in J_{-3}$ with $B_{\theta,j} \subset D \setminus D_\alpha$, and $ g_j = 0$ for $j \in J_{k}$ if $B_{\theta,j} \subset D \setminus D_\alpha$ and $B_{\theta,j} \cap B_{\theta,i} = \emptyset$ for all $i \in \cup_{l = -3}^{k-1}$ with $g_i \neq 0$.
	
	\medskip
		
	We conclude the proof of this step provided that $\bar g$  also satisfies \eqref{comp.cond.average}. We begin to observe that, since we extended $g$ to zero in $D_{\bar r } \setminus ( D_r \cup E^\eps)$ and $g \in L^q_0(D_r)$, we have
	\begin{align*}
	\int_{D_{\bar r}} \bar g &= \int_{ D_r} g + \int_{E_{-3}} \bar g = \int_{E_{-3}} \bar g .
	\end{align*}
	Furthermore, by \eqref{comp.cond} and the definition \eqref{remove.annuli} of the sets $E_k$ it also holds that
	\begin{align*}
		\int_{D_{\bar r}} \bar g = \int_{E_{-2}}\bar g + \sum_{j \in J_{-3}} \int_{B_j \setminus E_{-2}} \bar g - \int_{A_j \cap E_{-2}} \bar g  =  \int_{E_{-2}} \bar g.
	\end{align*}
	We may thus iterate the previous procedure and obtain that 
	\begin{align*}
	\int_{D_{\bar r}} \bar g = \int_{E_{\km + 1}} \bar g = 0,
	\end{align*}
 since $E_{k_{\max} + 1} = \emptyset$ (cf. \eqref{remove.annuli}).

	\bigskip	
	
	\noindent {\bf Step 3. Conclusion}	
	We begin by arguing that $\bar g$ defined in the previous step has $L^q$-norm that remains comparable to the one of $g$. More precisely, we claim that
	\begin{align}
		\label{est.g}
		\|\bar g\|^{q}_{L^{q}(D_{\bar r} )} \lesssim \|g \|^{q}_{L^{q}(D_r \setminus E)}.
	\end{align}
 We set $J = \bigcup_{k=-3}^{\km} J_k$  as in Lemma \ref{l.geometry}. Since by definition of $\bar g$ and \eqref{E_-3}, we have
	\begin{align}\label{norm.g.decomposition}
		\|\bar g\|^{q}_{L^{q}(D_{\bar r})} = \|g\|^{q}_{L^{q}(D_r \setminus E)} + \sum_{k=-3}^{\km}\sum_{z_j \in J_k} |\tilde E^{z_j}_k | |g_j|^{q},
	\end{align}
	\eqref{est.g} follows provided that for every $k= \km, \cdots, -3$ and all $z_j \in J_k$ we show that
	\begin{equation}
	\begin{aligned}
		\label{est.g_j}
		|\tilde E^{z_j}_k| |g_j| \leq (2k_{max} + 3)^{k+3} \|g\|_{L^{1}(\theta B_{\theta,j} \setminus E)},\\
		\#\{z_j  \in  J \colon x \in \theta B_{\theta,j}\} \leq k_{max} + 1 \ \ \ \text{for every $x \in D_{\bar r}$.}
	\end{aligned}
	\end{equation}
	Here,  by \eqref{definitions}, the set $\theta B_{\theta, j} = B_{\theta^2 \lambda_j \aeps \rho_j}(\eps z_j)$.

	\medskip
	
	Indeed, for every $k= \km, \cdots ,-3$ and $z_j \in J_k$ the first inequality in \eqref{est.g_j} and \eqref{hole.comparable} of Lemma \ref{l:E_-3} imply that
	\begin{align}
		|\tilde E^{z_j}_k| |g_j|^{q} \lesssim \frac 1 {|\tilde E^{z_j}_k|^{q - 1}} \|g\|^{q}_{L^{1}(\theta B_{\theta,j} \setminus E)}
		\lesssim   \|g\|^{q}_{L^{q}(\theta B_{\theta,j} \setminus E)}.
	\end{align}
	This, together with \eqref{norm.g.decomposition} and the second inequality in \eqref{est.g_j}  imply \eqref{est.g}.

	\medskip

	We prove the first inequality in \eqref{est.g_j} by induction over $k$. For $z_j \in J_{-3}$ we use $B_j \in E$ to deduce from \eqref{comp.cond}
	\begin{align}
	|\tilde E^{z_j}| g_j = \int_{A_j \cap E_{k+1} \backslash E} g.
	\end{align}
	Thus, \eqref{est.g_j} holds for $k=-3$.
	Assume that \eqref{est.g_j} holds for all $1 \leq l \leq k-1$ and consider $z_j \in J_{k}$. Then, by \eqref{comp.cond},
	\begin{equation}	
	\begin{aligned}
		\label{est.g_j.1}
		|\tilde E^{z_j}| |g_j| &\leq \int_{B_{\theta,j} \setminus E} |g| +  \sum_{l=-3}^{k-2} \sum_{z_i \in J_l} |B_{\theta,j} \cap \tilde E^{z_{i}}| |g_{i}| \\
		&\leq \|g\|_{L^{1}(B_{\theta,j} \setminus E)} +  \sum_{l=-3}^{k-2} \sum_{\substack{z_i \in J_l \\ B_i \cap B_{\theta,j} \neq \emptyset}} (2k_{max} + 3)^{k+1} \|g\|_{L^{1}(\theta B_{\theta,i} \setminus E)},
	\end{aligned}
	\end{equation}
	where we used that $\tilde E^{z_i} \subset B_i$.
	We observe that for $z_i \in J_l$, $l \leq k - 2$ with $ B_i \cap B_{\theta,j} \neq \emptyset$ we have $\theta B_{\theta,i} \subset \theta B_{\theta,j}$ since $\rho_i \ll \rho_j$ by \eqref{eq:order.J_k}. 
	Moreover, for every $x \in \theta B_{\theta,j}$, 
	\begin{align}\label{number.overlappings}
		\#\{z_{i} \in \cup_{l=-3}^{k-2} J_l: x \in \theta B_{\theta,i}\} \leq k + 1,
	\end{align}
	since, by \eqref{similar.size.apart}, $\theta B_{\theta,i_1} \cap \theta B_{\theta,i_2} = \emptyset$ whenever $z_{i_1} \neq z_{i_2} \in J_l$ for some 
	 $-3 \leq l \leq k-2$ .
	Using this in \eqref{est.g_j.1} yields the first line \eqref{est.g_j}. Furthermore, the same argument for \eqref{number.overlappings} implies the second line in \eqref{est.g_j}. This concludes the proof of \eqref{est.g}.

	\medskip
	
	We conclude the proof of this lemma provided that we show that the function $v:=v^{-3}$ defined in \eqref{eq:def.v_k} of Step 1 satisfies \eqref{Bogovski}. Note that, by Step 2, the extension $\bar g$ satisfies the compatibility conditions \eqref{comp.cond.average}-\eqref{comp.cond.0} and thus each $v^k$ constructed in \eqref{eq:def.v_k}  of Step 1 is well-defined. In particular, note that by the standard Bogovski Lemma (see Lemma \ref{lem.annulus.divergence} in the Appendix) there exists $v_0$ satisfying \eqref{step.zero}. Furthermore, by \eqref{est.g}, we also have that
		\begin{equation}
	\label{v_0.estimate}
			\|v_0\|_{H^{1}(D_{\bar r}} \lesssim \|g\|_{L^{q}( D_r \setminus E)}.
	\end{equation}	
	We claim that for each $k= \km+1, \cdots, -3$, the function $v^k$ satisfies
	\begin{equation} 
		\label{v^(k)}
	\begin{aligned}
		\dv v^{(k)} &= g \qquad \text{in } D \setminus E_k \\
		v^{(k)} &= 0 \qquad \text{in } H_g^\eps \cup E_{k}, \\
		\| v^{(k)} \|_{H^1} + \|v^{(k)}\|_{C^0} &\lesssim \| g \|_{L^{q}}.
	\end{aligned}
	\end{equation}
	Note that this, in the case $k=-3$, immediately yields Lemma \ref{Bogovski} for $v= v^{(-3)}$. It is easy to see by induction over $k$ and the definition \eqref{remove.annuli} of the sets $E_k$ that the first two identities above are satisfied. It remains to prove the third one. We argue by induction over $k$ and begin with $k= \km + 1$. 

\medskip

	By  (iv) of \cite[Lemma B.2]{GH19withoutpressure}, we have for the solution $v_j$ to \eqref{Stokes.annulus.divergence}
		\begin{align}
		\label{eq:est.stokes.annulus}
		\|v_j\|_{H^{1}(B_{\theta,j})} &\lesssim \|v_0\|_{H^1(B_{\theta,j})} + R_j^{\frac{d-2}{2}}\|v_0\|_{L^\infty}, \\
		\|v_j\|_{C^0} &\lesssim \|v_0\|_{C^0},
	\end{align}
	Moreover, by Lemma \ref{lem.annulus.divergence}, we can find a solution  $v_j$ to \eqref{Annulus.divergence} with
		\begin{equation} 		\label{Annulus.divergence.est}
	\begin{aligned}
		\|v_j\|_{H^1}  &\lesssim  \|v^{(k+1)}\|_{H^1(B_{\theta,j})} + \|\bar g\|_{L^{2}(B_{\theta,j})} 
		  R_j^{\frac{d-2}{2}} \Bigl(\|v^{(k+1)}\|_{C^0} + \|\dv v^{(k+1)} \|_{L^{q}(B_r)} + \|\bar g\|_{L^{q}}\Bigr), \\
		\|v_j\|_{C^0} &\lesssim \|v^{(k+1)}\|_{C^0} + \|\dv v^{(k+1)} \|_{L^{q}(B_r)} + \|\bar g\|_{L^{q}},
	\end{aligned}
	\end{equation}
	
	Thus, since the functions $v_j$ have disjoint support due to \eqref{good.set}, using \eqref{eq:est.stokes.annulus} and \eqref{est.g}
	\begin{align} \label{eq:v_km.infty}
		\|v^{(k_{max}+1)}\|_{C^0} \lesssim \|g\|_{L^{q}}
	\end{align}
	and 
	\begin{align} \label{eq:v_km.H1}
		\| v^{(k_{max}+1)} \|^2_{H^1} = \sum_{z_j \in n^\eps} \|v_j\|_{H^1}^2  
		& \lesssim \sum_{z_j \in n^\eps} \| v_0 \|_{H^1(B_{\theta,j})}^2 + \aeps \rho_j \|v_0\|_{L^\infty} \lesssim \|g\|_{L^{q}},
	\end{align}
	almost surely, for $\eps$ small enough. This concludes the proof of \eqref{v^(k)} for $v^{(\km +1)}$.  
	
Let us now assume that  \eqref{v^(k)} holds for some $k + 1$. Then, using that $ |\dv v^{(k+1)}| \leq |g|$ pointwise together with the estimates for $v^{(k+1)}$,
we get the estimate in \eqref{v^(k)} analogously as we obtained the estimates for $v^{(k_{max}+1)}$. 

\medskip

We conclude the proof by observing that each $v^{(k)} \in H^1_0(D)$, since we only changed $v^{(k+1)}$ in $B_{\theta,j}$ for holes that are in a cluster that overlaps with $D_{\alpha'}$. These balls are contained in $D$ by an argument analogous to the one at the end of Step 2 in the proof of \cite[Lemma 2.5]{GH19withoutpressure}. \null\hfill\qedsymbol

%% file: appendix.tex
\begin{proof}[Proof of Lemma \ref{l.oscillating.test.functions}]
The proof of this lemma follows by simply observing that the construction of the operator $R_\eps$ in \cite[Lemma 2.5]{GH19withoutpressure} does not strictly require 
that the function $\varphi$ is compactly supported. Morevoer, from a careful look to the construction of $R_\eps$ and the properties of $H^\eps$ in Lemma \ref{l.geometry} it 
is immediate to infer that $R_\eps \varphi=0$ in $E^\eps$. 
\end{proof}

\bigskip

\begin{proof}[Proof of Lemma \ref{l:E_-3}]
We begin by showing \eqref{Cap.E^eps}. By definition \eqref{E^eps} and \eqref{D_b} of Lemma \ref{l.geometry}, we have that $E^\eps \setminus H_\eps \subset D^\eps_b$.
Hence, the sub-additivity of the harmonic capacity  and Lemma \cite[Lemma C.2]{GH19withoutpressure} yield
	\begin{align}
		\mathrm{Cap}( E^\eps \setminus H_\eps) \leq \sum_{z_j \in J^\eps} \mathrm{Cap}\Bigl( B_{\Lambda \aeps \rho_j}(\eps z_j)\Bigr)  \lesssim \eps^d \sum_{z_j \in J^\eps} \rho_j^{d-2} \to 0
	\end{align}
	almost surely as $\eps \to 0$.
	

\medskip

 We now turn to \eqref{char.E^z_j}: Let us recall the definition of the sets $J, J_{-3}, \cdots J_{\km}$ of Lemma \ref{l.geometry}. We fix $k= -3, \cdots, \km$. For every
 $z_j\in J_k$, let $B_j$ be as in \eqref{definitions}. Then, definitions \eqref{E^eps}, \eqref{D_b} and property \eqref{safety.layer} of Lemma \ref{l.geometry} imply
 that $E \cap B_j= E_{-3} \cap B_j$. Moreover, by \eqref{E_k^z_j}--\eqref{def.E.z} we also have $E^{z_j} \subset B_j$. Hence, identity \eqref{E_-3} implies
	\begin{align}
		E^{z_j} = E \cap B_j \setminus \bigcup_{\substack{i \in J \\ i \neq j}} E^{z_i}.	
	\end{align}	
On the one hand, the definition of $E^{z_i}$ yields $B_j \cap E^{z_i} = \emptyset$ for $i \in J_l$, $l > k$. On the other hand, $B_j \cap E^{z_i} = \emptyset$ for 
$ j \neq i \in J_k \cup J_{k-1}$ by  \eqref{similar.size.apart}. These and the above identity immediately imply \eqref{char.E^z_j}.

\medskip

We turn to the proof of \eqref{balls.in.annulus}. As before, we fix $k$ and $z_j \in J_k$. Let $A_j$ be as in \eqref{definitions}. Since by \eqref{D_b} and \eqref{safety.layer},
we have $H_g \cap A_j = \emptyset$, identity \eqref{balls.in.annulus} immediately follows provided we argue that for all $z_i \in J_l$ with $l \geq k - 1$,
\begin{align}\label{balls.in.annulus.aux}
	E^{z_i} \cap A_j \cap E_{k+1} = \emptyset.
\end{align}
For $z_i = z_j$ this follows directly from the definition of $E^{z_j}$. For $l = k,k-1$ this is implied by  \eqref{similar.size.apart}.
Finally, for $ l \geq k+1$ identity \eqref{balls.in.annulus.aux} is obtained from 
\begin{align}
 E_{k+1} = \bigcup_{l \geq k+1} {\bigcup_{j \in J_l}} E^{z_i}_l \supset \bigcup_{l \geq k+1} {\bigcup_{j \in J_l}} E^{z_i},
\end{align}
which relies on \eqref{E_-3}. This establishes  \eqref{balls.in.annulus.aux}, as well as  \eqref{balls.in.annulus}.

%

\medskip

To conclude the proof of this lemma, it remains to show property \eqref{hole.comparable}. To do so, we need the following two ingredients which are contained in 
\cite{GH19withoutpressure}. We recall that the dilation parameters $\lambda_j^\eps$ of Lemma \ref{l.geometry} may be written as $\lambda_j^\eps = \theta^2 \tilde \lambda_j^\eps$ for some
$\theta \geq 1$ \cite[Step 2 of Lemma 2.3]{GH19withoutpressure}. Then, for all $-3 \leq k \leq k_{\max}$ and all $z_j \in J_k$,
we have
\begin{align}
	\label{eq:char.J_k}
	B_{\theta \tilde \lambda_j^\eps \aeps \rho_j}(\eps z_j) \not \subset E_{k+1},
\end{align}
and
\begin{align}
	\label{eq:E^z_j.contains.ball}
	B_{\aeps \theta^{3/2} \tilde \lambda_j^\eps \rho_j}(\eps z_j) \subset E^{z_j}.
\end{align}
Indeed, \eqref{eq:char.J_k} is the same as \cite[equation (3.32)]{GH19withoutpressure}. Moreover, for $\theta^{3/2}$ replaced by $\theta$, \eqref{eq:E^z_j.contains.ball} is just \cite[equation (3.49)]{GH19withoutpressure}.
The proof of \cite[equation (3.49)]{GH19withoutpressure} holds indeed also when replacing the parameter $\theta$ by any $\gamma$ such that $\gamma < \theta^2$.

\medskip

Let $k\in \{-3, \cdots, \km \}$ and $z_j \in J_k$ be fixed. If  $E^{z_j} \cap E_{k+1} = \emptyset$, then \eqref{hole.comparable}  directly follows from \eqref{eq:E^z_j.contains.ball}.
If, otherwise, $E^{z_j} \cap E_{k+1} \neq \emptyset$ then, by definition, also $B_j \cap E_{k+1} \neq \emptyset$. Hence, there exists $z_i \in J_l$, $l \geq k+1$ such
that $B_j \cap B_i \neq \emptyset$. This, together with \eqref{eq:char.J_k} and the definition of $E_{k+1}$, implies that there also exists $z_i \in J_l$, $l \geq k+1$
such that 
\begin{align}\label{minimality}
	A_i \cap B_{\theta \tilde \lambda_j^\eps \aeps \rho_j}(\eps z_j)  \neq \emptyset.
\end{align} 
Let $l \geq k+1$ be the smallest index satisfying the above inequality and let us fix $z_i \in J_l$ as center of the annulus $A_i$. By \eqref{similar.size.apart}, we necessarily 
have that $l \geq k+2$. We now claim that it suffices to argue
\begin{align}
\label{unique.parent.contradiction}
	A_i \cap B_{\theta^{3/2} \tilde \lambda_j^\eps \aeps \rho_j}(\eps z_j) \subset  E^{z_j} \backslash E_{k+1}.
\end{align}
This, indeed, immediately yields that
\begin{align}\label{inclusion.measure}
|E^{z_j} \backslash E_{k+1}| \geq |A_i \cap B_{\theta^{3/2} \tilde \lambda_j^\eps \aeps \rho_j}(\eps z_j)|.
\end{align}
Note that the radius of the ball in \eqref{unique.parent.contradiction} has been dilated by the factor $\theta^{\frac 12}$ with respect to the radius in \eqref{minimality}.
Furthermore, the fact that $l \geq k+2$ implies by \eqref{eq:order.J_k} in Lemma \ref{l.geometry} that $\rho_i \gg \rho_j$. These two considerations, together with 
\eqref{minimality}, imply that $|A_i \cap B_{\theta^{3/2} \tilde \lambda_j^\eps \aeps \rho_j}(\eps z_j)| \gtrsim |B_j|$ and thus allow to conclude \eqref{hole.comparable} from 
\eqref{inclusion.measure}. 

\bigskip

It thus remains only to prove \eqref{unique.parent.contradiction}. By \eqref{eq:E^z_j.contains.ball}, it suffices to show
\begin{align}
A_i \cap B_{\theta^{3/2} \tilde \lambda_j^\eps \aeps \rho_j}(\eps z_j) \cap E_{k+1} = \emptyset.
\end{align}
We prove this by contradiction. Let us assume that the above identity does not hold. Then, by definition of $E_{k+1}$ and \eqref{similar.size.apart},
there exists $z_m \in J_m$ with $k+2 \leq m \leq l - 2$ such that
\begin{align}
	A_i \cap B_{\theta^{3/2} \tilde \lambda_j^\eps \aeps \rho_j}(\eps z_j) \cap B_m \neq \emptyset.
\end{align}
However, using again that $k+2 \leq m$ implies $\rho_m \gg \rho_j$ by \eqref{eq:order.J_k}, this means
\begin{align}
		 B_{\theta \tilde \lambda_j^\eps \aeps \rho_j}(\eps z_j) \subset B_{m,\theta}.
\end{align}
Hence, either
\begin{align}
	A_m \cap B_{\theta \tilde \lambda_j^\eps \aeps \rho_j}(\eps z_j)  \neq \emptyset,
\end{align}
or 
\begin{align}
	B_{\theta \tilde \lambda_j^\eps \aeps \rho_j}(\eps z_j) \subset B_{m}.
\end{align}
The first case immediately contradicts the fact that $l$ is the minimal index for which we have \eqref{minimality}. In the second case, \eqref{eq:char.J_k} and the definition
of $E_{k+1}$ imply the existence of another index $\bar m$, with $ k+2 \leq \tilde m \leq m -2$, for which we may find and $z_{\tilde m} \in J_{\tilde m}$ such that 
\begin{align}
	A_{\tilde m} \cap B_{\theta \tilde \lambda_j^\eps \aeps \rho_j}(\eps z_j)  \neq \emptyset.
\end{align}
This as well, contradicts the minimality of $l$. The proof of \eqref{unique.parent.contradiction} is hence complete.
\end{proof}

\bigskip

Below we summarize some standard results for the solutions to the Stokes equation in annular and exterior domains (see, e.g. \cite{Galdi, AllaireARMA1990a}).  
These are extensively used in the proof of Lemma \ref{lem:Bogovski}.
\begin{lem}
	\label{lem.annulus.divergence}
	Let $q>d$ and let $0 < r < 1$, $\theta > 1$, $B_r := B_r(0)$, $B_{r \theta} := B_{r \theta}(0)$, $A_{r,\theta} := B_{r \theta} \setminus B_r$.
	Assume $g \in L^{q}(B_{r \theta})$ and $v \in H^1(B_{r \theta}) \cap C^0(B_{r \theta})$ with $\dv v \in L^{q}(B_r)$ satisfy
	\begin{align}
		 \int_{A_{r,\theta}}  g + \int_{\partial B_{r}} v \cdot \nu = 0.
	\end{align}
	Then, there exists $u \in H^1_0(B_\theta) \cap C^0(B_\theta)$ solving
	\begin{align}
		\label{Annulus.divergence.lemma}
		\begin{cases}
		\dv u = g \qquad &\text{in } A_{r,\theta}  \\
		u = 0  \qquad &\text{on } \partial B_{r \theta} \\
		u = v \qquad &\text{in }B_r,
		\end{cases}
	\end{align}	
	with
	\begin{align}
		\|u\|_{H^1} &\leq C \|v\|_{H^1} + \|g\|_{L^2} + r^{\frac{d-2}{2}}(\|v\|_{C^0} + \|\dv v \|_{L^{q}(B_r)} + \|g\|_{L^{q}})), \\
		 \|u\|_{C^0} &\leq C \|v\|_{C^0} + \|\dv v \|_{L^{q}(B_r)} + \|g\|_{L^{q}}).
	\end{align}
	with $C= C(\theta,d,q)$.
\end{lem}

\begin{proof}
	We will define $u = u_1 + u_2$, where $u_1$ solves
	\begin{align}
		\label{u_1}
		\begin{cases}
		\dv u_1 = g \qquad &\text{in } A_{r,\theta}  \\
		\dv u_1 = \dv v \qquad &\text{in } B_r  \\
		u_1 = 0  \qquad &\text{on } \partial B_{r \theta}, \\
	\end{cases}
	\end{align}	
	and  $u_2$ is the solution to 
	\begin{align} 		\label{u_2}
	\begin{cases}
		-\Delta u_2 + \nabla p = 0 \qquad &\text{in } A_{r, \theta}  \\
		\dv u_2 = 0 \qquad &\text{in } A_{r,\theta}  \\
		u = 0  \qquad &\text{on } \partial B_{r \theta} \\
		u = v - u_1 \qquad &\text{in }B_r.
		\end{cases}
	\end{align}
	As it is well known (see e.g. \cite{Galdi}[Theorem 3.1]), the first problem has a solution with 
	\begin{align}
		\|u_1\|_{H^1} &\lesssim \|\dv v \|_{L^{2}(B_r)} + \|g\|_{L^{2}}, \\
		\|u_1\|_{W^{1,q}} &\lesssim \|\dv v \|_{L^{q}(B_r)} + \|g\|_{L^{q}}.
	\end{align}
	By Sobolev inequality, 
	\begin{align}
		\|u_1\|_{C^0} \lesssim\|\dv v \|_{L^{q}(B_1)} + \|g\|_{L^{q}}.
	\end{align}
	Using \cite[Lemma B.1]{GH19withoutpressure} rescaled with $r$ for the solution to \eqref{u_2}, we find
	\begin{align}
		\|\nabla u_2\|_{L^2} &\lesssim \| \nabla (v - u_1)\|_{L^2} + \frac 1 r \|v - u_1\|_{L^2} \lesssim \| \nabla v\|_{L^2} + \|\nabla u_1\|_{L^2} + r^{\frac {d-2}{2}} 
		\|v - u_1\|_{C^0} \\
		&\lesssim \| \nabla v\|_{L^2} + \|g\|_{L^{2}} + r^{\frac {d-2}{2}} \Bigl(\|v\|_{C^0}\| +  \|\dv v \|_{L^{q}(B_1)} + \|g\|_{L^{q}}\Bigr),
	\end{align}
	and
	\begin{align}
		\|u_2\|_{C^0}\|  \lesssim \|v - u_1\|_{C^0}\| \lesssim \|v\|_{C^0}\| +  \|\dv v \|_{L^{q}(B_1)} + \|g\|_{L^{q}}.
	\end{align}
	Combining theses inequalities for $u_1$ and $u_2$ (and the Poincare inequality) yields the desired estimate for $u$.
\end{proof}